\documentclass[11pt]{amsart}
\usepackage{fullpage}
\usepackage{nicefrac}
\usepackage{enumerate}
\usepackage{amssymb,mathrsfs,mathtools,amsthm,amscd}
\usepackage[english]{babel}
\usepackage[utf8]{inputenc}
\usepackage[T1]{fontenc}

\usepackage{lmodern}
\usepackage[normalem]{ulem}

\usepackage{tikz}
\usetikzlibrary{patterns}
\usetikzlibrary{matrix}
\usepackage{xcolor}\usepackage[all]{xy}

\usepackage{enumitem}
\usepackage[colorlinks=true]{hyperref}
\usepackage{wasysym}

\usepackage{array}

\usepackage{alltt}
\usepackage{fancyhdr}

\usepackage{multicol}
\usepackage{cite}
\def\R{\mathbb{R}}
\def\C{\mathbb{C}} 
\def\N{\mathbb{N}} 
\def\V{\mathbb{V}} 
\def\Z{\mathbb{Z}} 
\def\P{\mathbb{P}}

\def\O{\mathscr{O}}

\def\ep{\varepsilon}
\def\w{\wedge}
\newcommand{\dr}{\partial}
\newcommand{\drb}{\overline{\partial}}
\newcommand{\mc}{\mathcal}
\let\leq\leqslant \let\geq\geqslant

\DeclareMathOperator{\rk}{rk} 

\DeclareMathOperator{\Ric}{Ric} 

\DeclareMathOperator{\lex}{\le_{exc}}
\DeclareMathOperator{\supp}{Supp}
\let\div\relax
\DeclareMathOperator{\div}{div}

\newcommand{\clog}[2]{\Omega_{#1}(\log #2)}

\def\rank{\mathrm{rank}}

\usepackage{dsfont}
\newcommand{\un}{\mathds{1}}

\newtheorem{Theo}{Theorem}[section]
\newtheorem{lemme}[Theo]{Lemma}
\newtheorem{Prop}[Theo]{Proposition}

\newtheorem{cor}[Theo]{Corollary}

\newtheorem{Ex}{\it Example\/}

\theoremstyle{definition}
\newtheorem{Def}[Theo]{Definition}
\newtheorem{remark}[Theo]{\it Remark\/}

\title{Nevanlinna second main theorem and tautological inequality for parabolic manifolds}
\author[]{Clara Dérand}
\thanks{Université de Lorraine, CNRS, IECL, F-54000 Nancy, France}
\address{C. Dérand: Institut Elie Cartan de Lorraine (IECL)\\ Centre de Recherche en Automatique de Nancy (CRAN)}
\email{clara.derand@univ-lorraine.fr}
\date{}
\begin{document}

\begin{abstract}
We obtain a \emph{Second Main Theorem type inequality} for holomorphic maps $f:M\to X$, where $M$ is a parabolic manifold and $X$ is smooth projective with $\dim M\leq \dim X$.
We also derive a \emph{parabolic Tautological inequality} for smooth logarithmic pairs.
\end{abstract}
\maketitle
\section{Introduction}
A crucial result concerning value distribution of meromorphic mappings is the \emph{Second Main Theorem}, established by R. Nevanlinna in 1925 \cite{Nev25}.
\begin{Theo}
    Let $f:\C\to\P^1(\C)$ be a non-constant holomorphic map. Then for any distinct points $a_1,\dots,a_q\in\P^1(\C)$ and for any $\ep>0$,
    \[(q-2)T_f(r)\lex\sum_{j=1}^qN_f(a_j,r)+\ep T_f(r),\]
    where $\lex$ means that the inequality holds outside a Borel subset of finite measure.
\end{Theo}
This inequality provides an estimate of the size of the preimage $f^{-1}({a_1,\dots,a_q})$ by a term quantifying the `growth' of the map $f$, the \emph{characteristic function} $T_f(r)$. It can be regarded as an extension of the Fundamental Theorem of Algebra for polynomials to the transcendental setting.

In the following hundred years, many variations of the Second Main Theorem have been obtained in the more general framework of value distribution of entire curves in projective varieties. Some work has also been done for holomorphic mappings $f:M\to X$ into a projective variety $X$, considering a more general source manifold $M$. An overview of the main achievements in Nevanlinna theory can be found in \cite{NW14}.
\subsection{Parabolic manifolds}
In this paper, we shall study the case where $M$ is a \emph{parabolic manifold}: it is a non-compact complex manifold equipped with a \emph{special exhaustion function}, i.e. an plurisubharmonic exhaustion function $\sigma:M\to\R_+$, solution of the complex homogeneous Monge-Ampère equation.

The extension of Nevanlinna theory to parabolic manifolds has first been studied by Stoll \cite{Sto77}. These objects provide indeed a suitable framework to generalise value distribution theory of holomorphic maps in several complex variables, since this class includes notably $\C^m$ and holomorphic coverings of $\C^m$.

After Stoll's work, the case of parabolic Riemann surfaces has been investigated by several works, see \cite{Sto80} \cite{Sto83}\cite{PS14} \cite{BB20} \cite{HR21} \cite{CHSX23}. In higher dimension, the best known cases are for $M$ affine (strictly parabolic case) or a finite cover of an affine variety, and when $\dim X\le \dim M$ (case of a dominant meromorphic map) see also   \cite{AN91}\cite{Ai01} \cite{RW05}\cite{GPR13}\cite{Ru14}. However, the general problem has not aroused great interest in the past years, probably because the theory of parabolic manifolds suffer from the lack of other concrete examples and applications.

Let us introduce some standard notation of Nevanlinna theory.
Denote $B(r)=\{\sigma\le r^2\}$, $S(r)=\{\sigma=r^2\}$. For any $d$-dimensional analytic subvariety $Z\subset X$, we define:
\begin{equation*}
\left\lbrace
\begin{aligned}
  n(Z,r)&=\int_{B(t)\cap Z}(dd^c\sigma)^d\\N(Z,r)&=\int_{r_0}^rn(Z,t)\frac{dt}{t^{2m-1}}
\end{aligned}
\right.
\end{equation*}
The counting function $N(Z,r)$ describes the size of the intersection of $Z$ and $f(M)$. 

For a smooth $(1,1)$-form $\omega$ on $M$, we define the \emph{Nevanlinna characteristic function} by:
\[T_{\omega}(r)=\int_{r_0}^r\frac{dt}{t^{2m-1}}\int_{B(t)}\omega\wedge(dd^c\sigma)^{m-1}.\]
For a divisor $D\subset M$, one can consider $$T_{f,D}=T_{f^*c_1(\mc{O}_X(D),h)}$$ for any smooth metric $h$ (it can be shown that this integral is independent of the metric chosen, up to a bounded term).
This quantity describes the growth of the image of $f$ relatively to the pseudo-metric given by the Chern class of $\mc{O}_X(D)$.
\subsection{Main results}
The core of this work (see section \ref{sec:smt}) consists in a Second Main Theorem-type inequality for meromorphic maps $f:M\to X$, where $M$ is a parabolic manifold of dimension $m$ and $X$ is a smooth projective variety of dimension $n\geq m$, which extends the SMT of \cite{BB20}. 
\begin{Theo}[= Theorem \ref{thm:SMT-parab}]
Let $X$ be a smooth projective variety of dimension $n$ and $D$ a divisor on $X$. 
Endow $X\setminus D$ with a smooth hermitian pseudo-metric $h$ whose holomorphic sectional curvature is bounded above by a negative constant $-\gamma$ and holomorphic bisectional curvature is non positive. Assume that the degeneracy set $\Sigma_h$ of $h$ is a thin analytic subset.

Let $(M,\sigma)$ be a parabolic manifold of dimension $m\le n$ and $f:M\to X$ a non-constant holomorphic map such that $f(M)\not\subset D$ and $f(M)\not\subset\Sigma_h$. Denote by $\omega$ the $(1,1)-$form associated to the pseudo-metric $f^*h$ on $M\setminus f^{-1}(D)_{red}$. Then $\omega$ is locally integrable on $M$ and we have the following Second Main Theorem-type inequality:
$$\gamma T_{[\omega]}(r)\lex2\pi N^{[1]}_{f,D}(r)-\Ric_{\sigma}(r)+O(\log T_{[\omega]}(r)+\log r).$$
\end{Theo}
In section \ref{sec:proof SMT for VHS}, we immediately derive higher-dimensional versions of other results of the paper \cite{BB20}.
\begin{Theo}[= Theorem \ref{thm:SMT for VHS}]

Let $X$ be a smooth projective complex algebraic variety of dimension $n$ and $\mathbb{V} = (\mc{L}, \mc{F}^{\bullet}, h)$ be a variation of complex polarised Hodge structures defined on the complement of a normal crossing divisor $D \subset X$. Assume that $\mc{L}$ has unipotent monodromies around the irreducible components of $D$. We denote by $\bar{\mc{F}^p}$ the canonical Deligne-Schmid extension of $\mc{F}^p$ to $X$ for any integer $p$ and by $\bar{L}_{\mathbb{V}} = \otimes_p \det \bar{\mc{F}^p}$ the canonical extension of the Griffiths line bundle of $\mathbb{V}$.
Let $(M,\sigma)$ be a parabolic manifold of dimension $m\le n$ and $f:M\to X$ be a non-constant holomorphic map such that \(f(M)\not\subset D\). Then the first Chern form of the line bundle $ L_{\mathbb{V}}$ equipped with the hermitian metric induced by $h$ extends as a current \( [L_{\mathbb{V}}  ] \) on $M$, and for any ample line bundle $A$ on $X$, one has
  \[T_{f, \bar{L}_{\mathbb{V}}}(r)  \lex \frac{w^2 \cdot \rk \mc{L}}{4}  \left( - \frac{1}{2\pi}\Ric_\sigma(r) +  N_{\Sigma}(r)\right) +O(\log r+\log( T_{f,A}(r))).\]

\end{Theo}
In section \ref{sec:tauto}, we present a result similar to \emph{McQuillan's Tautological inequality} for smooth logarithmic pairs in the parabolic setting.

\begin{Theo}[= Theorem \ref{thm:tauto}]
Let $f:M\to X$ be a meromorphic map from an $m$-dimensional parabolic manifold $(M,\sigma)$ to a smooth projective variety $X$ of dimension $n\geq m$. Assume that $f$ has maximal rank $m$. Denote by $\Ric_\sigma$ the Ricci function of $(M,\sigma)$.

Let $D\subset X$ be a divisor with at most normal crossings such that $f(M)\not\subset D$. We consider the projectivisation of the $m$-th exterior power of the logarithmic cotangent bundle $X_1=\P(\bigwedge^m\clog{X}{D})$ equipped with its tautological line bundle $\O_{X_1}(1)$. Denote by $f':M\dashrightarrow X_1$ the rational map induced by the morphism of sheaves $f^\ast \bigwedge^m\clog{X}{D}\to \bigwedge^m\clog{M}{(f^{-1}D)_{red}}$.
Then one has
    \[T_{f',\O_{X_1}(1)}(r)\lex N_{f,D}^{[1]}(r)-\frac{1}{2\pi}\Ric_\sigma(r)+O\left(\log T_{f,\omega}(r)\right).\]
\end{Theo}
\subsection*{Acknowledgements}
The author is grateful to Damian Brotbek and Damien Mégy for their precious help and advices, and their constant support during the past four years.

She would like to thank Erwan Rousseau for his enlightening explanations on the higher-dimensional tautological inequality proved in \cite{GPR13} and his suggestions.

She also thanks Julie Wang for their fruitful discussions on this work, and her many questions and advices that highly contributed to clarify the corresponding part of the author's thesis.
\section{Differential geometry of complex manifolds}
In this section, we recall some background on metrics and plurisubharmonic functions and give preliminary results that we will need in the rest of the paper.
\subsection{Metrics and curvature}
Let $X$ be a smooth complex variety of dimension $n$. A \emph{smooth hermitian pseudo-metric} on $X$ is a smooth map $h:TX\times_XTX\to\C$ such that for any $x\in X$, the restricted map \[h_x:T_xX\times T_xX\to\C\] is a non-negative hermitian form. We shall denote by \[\Vert\cdot\Vert_h:TX\to\R_+,\xi\mapsto\sqrt{h(\xi,\xi)}\] the associated pseudo-norm. We define the degeneracy set of $h$ by \[\Sigma_h=\{x\in X\,|\,\exists\, v\in T_xX\setminus\{0\}, \Vert v \Vert_{h,x}=0\}.\]

Consider holomorphic local coordinates $(z_1,\dots,z_n)$ defined over an open subset $U\subset X$. The pseudo-metric $h$ has the local expression $$h(z)=\sum_{i,j}h_{i,j}(z)dz_id\bar{z_j},$$where $h_{i,j}$ is the smooth function defined by
$$h_{i,j}(z)=h\left(\frac{\dr}{\dr z_i},\frac{\dr}{\dr z_j}\right).$$
We associate with $h$ a global $(1,1)$-form $\omega_h$, which can be defined in local coordinates $(z_1,\cdots,z_n)$ by $$\omega_h=\frac{i}{2}\sum_{i,j}h_{i,j}dz_i\w d\bar{z_j}.$$

Let $\Omega$ be a non-negative smooth $2n$-form (pseudo-volume form). 
On the open set $U$, one can define a non-negative smooth real function $\lambda$ by$$\Omega=\left(\frac{i}{2}\right)^n\lambda\prod_{1\leq j\leq n}(dz_j\w d\bar{z_j}).$$

The \emph{Ricci curvature} of $\Omega$ is
$$\Ric\Omega=-2\pi dd^c\log\lambda=-i\sum_{i,j}\frac{\dr^2\log\lambda}{\dr z_i\dr\bar{z_j}}dz_i\w d\bar{z_j}.$$Note that this expression makes sense only on the locus where $\Omega$ is a positive volume form. Under some conditions, though, it is possible to extend $\Ric\omega$ as a current on the whole $X$.

In particular, if $\omega$ is the $(1,1)$-form associated to a smooth hermitian pseudo-metric $h$, we set $\Ric\omega:=\Ric\omega^n$.

Let $D\subset X$ be a divisor and denote by $\Omega$ a smooth pseudo-volume form defined over the complement $X\setminus D$. Let $U$ be an open subset with holomorphic local coordinates $z_1,\dots,z_n$ and $s$ a holomorphic function such that $D\cap U=(s=0)$. As before, we denote by $\lambda$ the non-negative smooth function on $U$ such that $\Omega=\left(\frac{i}{2}\right)^n\lambda\prod_{1\leq j\leq n}(dz_j\w d\bar{z_j}).$
We say that $\lambda$ (or abusively $\Omega$) has \emph{logarithmic singularities} along $D$ on $U$ if 
\[\lambda(z)=O\left(\frac{1}{|s(z)|^2}\right).\]

If $\log\lambda$ is locally integrable, it defines a $(1,1)$-current $[\log\lambda]$. When the function $\log\lambda$ is locally integrable in any local chart, one says abusively that $\log\Omega$ is locally integrable and defines a global $(1,1)$-current on $X$, which we call extensively the \emph{Ricci curvature} of $\Omega$. Locally, one has
$$\Ric[\Omega]:=-2\pi dd^c[\log\lambda].$$ Here the differential is to be understood in the sense of distributions, that is, for any smooth $(n-1,n-1)$-form $\phi$ with compact support,
$$<\Ric[\Omega],\phi>=\int_X(-2\pi\log\lambda)dd^c\phi.$$

\subsection{Plurisubharmonic functions}
Plurisubharmonic functions were introduced by Lelong and Oka in order to study holomorphically convex sets. We refer to \cite{Lel68} for more details. They are the higher-dimension generalisation of subharmonic functions over $\C$.

Recall that a function $u:U\to [-\infty,+\infty[$ defined on an open subset $U\subset \C$ is \emph{subharmonic} if it satisfies the \emph{mean value inequality}
$$u(z)\leq\frac{1}{2\pi}\int_0^{2\pi}u(z+re^{i\theta})d\theta$$for any $z\in U$ and $r>0$ such that $r<d(a,\partial U)$.
\begin{Def}
A function $u:\Omega\to[-\infty,+\infty[$ defined over an open subset $\Omega\subset\C^n$ is said to be plurisubharmonic if:
\begin{enumerate}
    \item $u$ is upper-semi-continuous;
    \item for every complex line $\ell\subset\C^m$, the restriction $u|_{\Omega\cap\ell}$ is subharmonic. This means that for all $a\in\Omega$ and $\xi\in\C^n$ such that $|\xi|<d(a,\partial\Omega)$,
    $$u(a)\leq \frac{1}{2\pi}\int_0^{2\pi}u(a+e^{i\theta}\xi)d\theta.$$
\end{enumerate}
The space of plurisubharmonic (\emph{psh}) functions over $\Omega$ is denoted by $\mathrm{Psh}(\Omega)$.
\end{Def}
Recall that a function $v:U\subset\C\to[-\infty,+\infty[$ is subharmonic if it satisfies the sub-mean value property at any point $z\in U$:
$$v(z)\leq\frac{1}{2\pi}\int_0^{2\pi} v(z+re^{i\theta})d\theta \mbox{ whenever }D(z,r)\subset U.$$
It follows that property 2 above is equivalent to the fact that $u$ satisfies the inequality $$u(z)\leq\frac{1}{2\pi}\int_0^{2\pi}u(z+e^{i\theta}\xi)d\theta$$for all $z\in\Omega$ and $\xi\in\C^m$ such that $|\xi|<d(z,\Omega^{\complement})$.

Plurisubharmonic functions have several properties.
\begin{enumerate}
\item Assume that $u\not\equiv-\infty$ on every connected component of $\Omega$. Then $u$ is locally integrable on $\Omega$.
 \item For any decreasing sequence of functions $u_k\in\mathrm{Psh}(\Omega)$, the limit $u=\lim\limits_{k\to\infty}u_k$ is psh.
 \item Let $u\in \mathrm{Psh}(\Omega)$ such that $u\not\equiv -\infty$ on any connected component of $\Omega$. Assume that $(\rho_{\ep})$ is a family of smoothing kernels. Then the $u\star\rho_{\ep}$ are smooth and psh on $\Omega_{\ep}=\Omega\cap B(0,\ep)$, the family $(u\star\rho_{\ep})$ is non-decreasing in $\ep$ and converges to $u$ as $\ep\to0$.

\end{enumerate}

An function $u\in\mc{C}^2(\Omega,\R)$ is plurisubharmonic if and only if its complex Hessian \[\left(\frac{\dr^2u}{\dr z_j\dr\bar{z}_k}(a)\right)_{1\leq j,k\leq n}\] 
is a semi-positive matrix at every point $a\in\Omega$. Passing to the limit (with point 3) above), we see that this is actually still true for general psh functions.
\begin{Theo}
    Let $u\in\mathrm{Psh}(\Omega)$ such that $u\not\equiv-\infty$ on any connected component of $\Omega$. For all $\xi\in\C^n$, the distribution $$Hu(\xi):=\sum_{1\leq j,k\leq n}\frac{\dr^2u}{\dr z_j\dr \bar{z}_k}\xi_j\bar{\xi}_k$$is a positive measure. Conversely, for any $v\in\mathscr{D}'(\Omega)$ such that $Hv(\xi)$ is a positive measure for any $\xi\in\C^n$, there exists a unique psh function $u$ on $\Omega$ such that $v$ is the distribution associated to $u$.
\end{Theo}

This notion can naturally be generalised to complex manifolds. Let $X,Y$ be two complex manifolds. Consider a holomorphic map $\varphi:X\to Y$ and a holomorphic function $u\in\mathscr{C}^2(Y,\R)$.
Since $\dr\drb(u\circ\varphi)=\varphi^*\dr\drb u$, we have
$$H(u\circ\varphi)(z,\xi)=Hu(\varphi(z),d\varphi(z)\cdot\xi),$$
so that the Hermitian form $Hu$ on $TX$ is independent of the choice of local coordinates.
This allows us to define plurisubharmonic functions on complex manifolds.

A map $u:M\to\R\cup\{-\infty\}$ is said \emph{plurisubharmonic} if for any holomorphic chart $\phi:V\subset\C^m\stackrel{\sim}{\longrightarrow} U\subset M$, the function $u\circ\phi:V\to\R\cup\{-\infty\}$ is plurisubharmonic.

We now give two basic results relating plurisubharmonic functions and Ricci currents which will be used further in chapter 2.
\begin{lemme}
Let $\psi$ be a plurisubharmonic function on a complex manifold $X$ which is of class $\mc{C}^{\infty}$ over $\{\psi\neq-\infty\}$.  Assume that $\psi$ is bounded from above. Then $\psi$ extends as a plurisubharmonic function on $X$. Moreover, the $(1,1)$-form $dd^c\psi$ is locally integrable on $\Delta$ and one has$$[dd^c\psi]\leq dd^c[\psi]$$in the sense of currents.
\end{lemme}
\begin{proof}
\normalsize
    The first point is a classical result in the theory of plurisubharmonic functions. It follows from Theorem 5.24, p. 45 of \cite{Dem_book}.

    Let $\eta$ be a smooth $(n-1,n-1)$-form with compact support in $X$ and $(\eta_n)$ a sequence of smooth $(n-1,n-1)$-form with compact support in $X\setminus\{\psi=-\infty\}$ converging to $\eta$ such that $0\leq \eta_n\leq\eta$.
    Then $[dd^c\psi](\eta_n)$ converges to $[dd^c\psi](\eta)$ by monotone convergence theorem.

    Moreover, $\eta-\eta_n\geq 0$ and $\psi$ is plurisubharmonic, so that 
    \[dd^c[\psi](\eta-\eta_n)\geq0.\]

    We obtain
    \[[dd^c\psi](\eta_n)=dd^c[\psi](\eta_n)\leq dd^c[\psi](\eta)\]
    and conclude by passing to the limit.
\end{proof}
\begin{lemme}
\label{ricci}
    Let $X$ be an $n$-dimensional complex manifold with a reduced divisor $D$. Let $h$ be a smooth pseudo-metric defined on $X\setminus D$ with associated $(1,1)$-form $\omega$. Assume that $\omega^n$ has at most logarithmic singularities along $D$ and that $\log\omega^n$ is locally integrable over $X$. If there exists a smooth $d$-closed $(1,1)$-form $\tilde{\omega}$ defined on $X$ such that $$\tilde{\omega}-\Ric\omega\geq0$$on $X\setminus(D\cup\Sigma_h)$, then:
    \begin{enumerate}
        \item both currents $[\Ric\omega]$ and $\Ric[\omega]$ are well-defined over $X$;
        \item one has $$[-\Ric\omega]\leq 2\pi[D]+(-\Ric[\omega]).$$
    \end{enumerate}
\end{lemme}
\begin{proof}
\normalsize
The statement is local, hence we can restrict to an open subset $U$. Fix holomorphic local coordinates $z_1,\dots,z_n$ on $U$ and a holomorphic function $s$ such that $D\cap U$ is defined by $(s=0)$. Let $\lambda$ be the non-positive smooth function on $U$ such that $\omega^n=\left(\frac{i}{2}\right)^n\lambda\prod_{1\leq j\leq n}(dz_j\w d\bar{z_j}).$

    Define a function by $\phi(z)=\log(|s(z)|^2\lambda(z))$. By assumption, $\phi$ is bounded from above on $U$. Moreover, $\phi$ is smooth outside $E=D\cup\Sigma_h$.
 Applying the $dd^c$-lemma to $\tilde{\omega}$, up to taking an open subset of $U$, there exists a smooth function $\alpha$ such that $\tilde{\omega}=dd^c\alpha$.

 By assumption, the function $\psi=\phi+\alpha$ is plurisubharmonic and bounded from above. Hence $dd^c\psi$ is locally integrable by the previous lemma, and so is $dd^c\phi$ outside $E$.
This shows that $[\Ric\omega]$ is well-defined.
The lemma also gives
\[[dd^c\alpha]+[dd^c\phi]\leq dd^c[\alpha]+dd^c[\phi].\]
Since $\alpha$ is $\mc{C}^\infty$, $[dd^c\alpha]=dd^c[\alpha]$, and Lelong-Poincaré's lemma applied to $D$ gives
\[[-\Ric\omega]=[dd^c\phi]\leq dd^c[\phi]=dd^c[\log|s|^2]+dd^c[\log\lambda]=[D]-\Ric[\omega].\]
 
\end{proof}
\section{Value distribution on parabolic spaces}
Consider the function $\sigma:z\in\C\mapsto|z|\in\R_+$. This is an \emph{exhaustion} of $\C$: $\sigma$ is unbounded and $\sigma^{-1}([0,r])=\overline{\Delta_r}$ is compact for any $r\geq0$. Moreover, the function $\log\sigma$ is harmonic. The pair $(\C,\sigma)$ is what we call a parabolic complex manifold. In this section, we generalise this concept and give properties and examples of parabolic spaces.
For more details and proofs, we refer to the standard book on this subject written by Wilhelm Stoll, \emph{Value distribution on Parabolic Spaces} \cite{Sto77}.
\subsection{Definitions}
Let $M$ be a complex manifold of dimension $m>0$. An exhaustion of $M$ is a proper smooth map $\sigma:M\to\R_+$. We define
\begin{itemize}
\item the respectively closed and open pseudoballs of radius $r$
\[M[r]=\{z\in M|\sigma(z)\leq r\}=\sigma^{-1}([0,r])\mbox{ and }M(r)=\sigma^{-1}([0,r));\]
\item the pseudo-sphere $M\langle r\rangle=M[r]\setminus M(r)$;
\item $M[s,r]=M[r]\setminus M(s)$ if $0\leq s\leq r$
\item $r_0=\min\{r\in\R_+|M[r]\neq\emptyset\}$.
\end{itemize}
Given $A\subset M$, we denote $A[r]=A\cap M[r]$ and do similarly for the other sets.

The exhaustion $\sigma$ is
\begin{itemize}
    \item \emph{central} is $M[0]$ has measure zero;
    \item \emph{strong} if $(dd^c\sigma)^m\not\equiv0$ on $M(r)$ for any $r>r_0$;
    \item \emph{(strictly) pseudo-convex} if $dd^c\sigma\geq0 (>0)$
    \item \emph{logarithmic pseudo-convex} if $dd^c\log\sigma\geq0$ on $M\setminus M[0]$.
\end{itemize}
A central and logarithmic pseudo-convex exhaustion is \emph{parabolic} if $(dd^c\sigma)^m\not\equiv0$ and $(dd^c\log\sigma)^m\equiv0$ on $M^+=M\setminus M[0]$. It is said strict if $dd^c\sigma>0$ on $M$.

Such a pair $(M,\sigma)$ with $\sigma$ parabolic is called a parabolic manifold.

In order to simplify the equations, we define:
\begin{align}
\alpha&=dd^c\sigma&\beta&=dd^c\log\sigma\\
\eta&=d^c\log\sigma\wedge(dd^c\log\sigma)^{m-1}&\nu&=d^c\sigma\wedge(dd^c\sigma)^{m-1}\\
\eta_r&=\eta|_{M\langle r\rangle}&\nu_r&=\nu|_{M\langle r\rangle}
\end{align}
\begin{lemme}
The following relations are satisfied on $M^+$:
\begin{align}
d\nu&=\alpha^m&d\eta&=\beta^m=0&\sigma^m\eta&=\nu\\
\sigma\alpha^m&=md\sigma\wedge\nu=m\sigma^md\sigma\wedge\eta\\
\zeta&=\int_{M\langle r\rangle}\eta_r \mbox{ is constant}\\
\int_{M(r)}\alpha^m&=\int_{M\langle r\rangle}\nu_r=r^{2m}\zeta
\end{align}
\end{lemme}

\begin{Ex}
Explicit examples of parabolic spaces are difficult to construct but we shall give some basic ones.
\begin{enumerate}
\item The most standard example is $(\C^m,\sigma)$ where $\sigma(z)=|z|^2$ for all $z$. One has $dd^c\sigma=\frac{i}{2\pi}\sum_jdz_j\w d\bar{z}_j>0$ and directly $dd^c\log\sigma=0$ on $\C^m\setminus\{0\}$. In this case, the pseudo-ball are simply the Euclidean balls.
\item Further, let now $\beta:M\to \C^m$ be a holomorphic covering of $\C^m$ by a complex manifold $M$ (that is, a proper surjective holomorphic map of strict rank $m$). Then $\beta^*\sigma$ is a parabolic exhaustion on $M$, where $\sigma(z)=|z|^2$. Moreover, $\zeta$ equals the sheet number of the covering.
\item In particular, any affine variety is parabolic.
\item A non-compact Riemann surface is parabolic, if and only if every bounded above subharmonic function is constant.
\item An non-affine example of a higher-dimensional parabolic manifold is given by \[(\C\setminus\Z)\times\C^m\] (\textit{cf} Stoll \cite{Sto77}, p. 85).

\end{enumerate}
\end{Ex}
We can also find in Stoll's book the following interesting example of the total space of a non-positive line bundle.
\begin{Prop}[see Prop. 10.15 in \cite{Sto77}]
Let $X$ be a compact complex manifold of dimension $m-1$ and $L\to X$ a non-positive holomorphic line bundle over $X$ endowed with a smooth hermitian metric $\kappa$ such that $\omega=c_1(L,\kappa)\le 0$. Assume that there exists a non-empty open subset $U\subset X$ on which $c_1(L,\kappa)<0$.
Then the function defined on the total space $M$ of $L$ by $\sigma(z)=|z|_{\kappa}^2$ is a parabolic exhaustion with $$\int_X\omega^{m-1}>0.$$
\end{Prop}
\begin{proof}
\normalsize
    It is straightforward to see that $\sigma$ is a central exhaustion with $dd^c\log\sigma=\pi^*\omega\geq0$, hence logarithmic pseudo-convex. Moreover, since $\dim X=m-1$, one has $\omega^m=0$, so that $(dd^c\log\sigma)^m=0$.
    Define then $\phi=\pi|_{B(1)}$ and $j:\Delta(0,1)\subset\C\to\phi^{-1}(\phi(x_0))$ by $j(z)=zx_0$, for some fixed $x_0\in S(1)$. Clearly $j$ is biholomorphic.
\end{proof}

\subsection{Jensen's formula}
Consider a parabolic manifold $(M,\sigma)$ of dimension $m$.

We have the following generalisation of Jensen's formula.
\begin{Prop}[Thm. 11.2, \cite{Sto77}]
Let $\phi:M\to[-\infty,+\infty[$ be a function which can locally on $M$ be written as a difference of two plurisubharmonic functions.
Then we have for $r>r_0$,
$$\int_{r_0}^r\frac{dt}{t}\int_{M(t)}dd^c\log\phi\wedge\beta^{m-1}=\int_{M\langle r\rangle}\phi\eta_r-\int_{M\langle r_0\rangle}\phi\eta_{r_0}=\int_{M\langle r\rangle}\phi\eta_r+O(1).$$
\end{Prop}
\begin{proof}
\normalsize
 To start with, remark that we can assume that $\phi$ is smooth (otherwise we can use a partition of unity). For each regular value $r$ of $\sigma$, the function $\phi$ is integrable with respect to the measure $\eta_r$ over the pseudo-sphere $M\langle r\rangle$.
 We have

 \begin{align*}
  \int_{r_0}^r\frac{dt}{t}\int_{M(t)}dd^c\phi\w\beta^{m-1}&=\int_{r_0}^r\frac{dt}{t}\int_{M(r)}\un_{M(t)}dd^c\phi\w\beta^{m-1}\\
  &=\int_{M(r)}\int_{\sigma(x)}^r\frac{dt}{t}dd^c\phi\w\beta^{m-1}\\
  &=\int_{M(r)}\log\frac{r}{\sigma}dd^c\phi\w\beta^{m-1}\\
  &=-\int_{M(r)}\phi dd^c\log\sigma\w\beta^{m-1}+\int_{M\langle r\rangle}(\phi d^c\log\sigma-\log(r/\sigma)d^c\phi)\w\beta^{m-1}
 \end{align*}
\end{proof}

\subsection{Holomorphic maps from a parabolic manifold}
We provide a summary of the main objects and basic results of value distribution theory for parabolic manifolds. For a more detailed presentation, we mainly refer to \cite{Sto77}, \cite{Sto80}, \cite{Sto83},\cite{SN66},\cite{NW14},\cite{AN91}, \cite{PS14} and \cite{BB20}.

Let $\beta$ be a current of type $(1,1)$ and order $0$ on $M$. We define the \emph{characteristic function of $\beta$} for $r>r_0$ by
$$T_{\beta}(r)=\int_{r_0}^r\frac{dt}{t^{2m-1}}\int_{M(t)}\beta\wedge\alpha^{m-1}.$$

For any analytic $(n-1,n-1)-$cycle $\Sigma$ in $M$, one has a well-defined integration current $[\Sigma]$ and we define the \emph{counting function with respect to $\Sigma$} by $N_{\Sigma}=T_{[\Sigma]}$.

Let $f:M\to X$ be a meromorphic map from $M$ to some projective variety $X$.
Consider a holomorphic line bundle $L$ on $X$ endowed with a Hermitian metric $h$, and a divisor $D$ on $X$ such that $f(M)\not\subset \supp D$, and write $(f^*D)_{red}=\sum_ia_iD_i$. Choose a section $s\in H^0(X,\mc{O}_X(D))$ such that $D=\div(s)$ and a Hermitian metric $h_D$ on $\O_X(D)$. 

We let for $q\ge 1$ and $r>r_0:$
\begin{align*}
T_{f,L}(r)&=T_{f^*c_1(L,h)}(r)\\
N^{[q]}_{f,D}(r)&=T_{D^{[q]}}(r), \mbox{ where }D^{[q]}=\sum_i\min(a_i,q)D_i\\
m_{f,D}(r)&=-\int_{M\langle r\rangle}\log (|s\circ f|_{h_D})\eta_r\ge0
\end{align*}
It is not difficult to see that the definitions of these functions are independent of the choice of metrics, up to a bounded term. 

The function $m_{f,D}$ is the \emph{proximity function}. It evaluates precisely how the image of $f$ approaches the divisor $D$ on the pseudo-sphere of radius $r$: $m_{f,D}$ goes to infinity as $f(M)$ gets closer to $D$ on $M\langle r\rangle$.

As in the case $M=\C$, these functions satisfy the First main theorem.

\begin{Theo}[\cite{Sto77}, Theorem 12.3]
    In the previous context, let $L$ be a holomorphic line bundle on $X$ and $D$ a divisor in the linear system $|L|$. Then for all $r>r_0$,
    \[T_{f,L}(r)=N_{f,D}(r)+m_{f,D}(r)+O(1).\]
\end{Theo}
The following lemmas (2.7 and 2.8 in \cite{BB20}) are direct consequences of Jensen's formula and hence they remain valid when $M$ is a higher-dimensional parabolic manifold. 
\begin{lemme}\label{thm:lemme1}
Let $X$ be a smooth projective variety. Let $L$ be a line bundle on $X$ and let $h_1$ be a smooth hermitian metric on $L$. Let $h_2$ be a singular metric on $L$. If the curvature current $c_1 (L, h_2 )$ is positive, then for any parabolic manifold $(M,\sigma)$ and any non-constant meromorphic map $f:M\to X$ one has
\[T_{f,c_1(L,h_2)}(r)\leq T_{f,c_1(L,h_1)}(r)+O(1).\]

\end{lemme}
\begin{lemme}\label{thm:lemme2}
    Let $X$ be a smooth projective variety. Let $L$ be a line bundle on $X$ and let $h_1$ be a smooth hermitian metric on $L$. Let $h_2 = e^{-\phi} h_1$ be a singular metric on $L$. Let $D$ be an effective divisor on $X$ and $s_D \in H^0 (X, \O_X (D))$ a global section defining $D$. Let $\Vert \cdot\Vert_D $ be the norm associated to a hermitian metric on the line bundle $\O_X(D)$. If locally on $X$ there exist $\beta,C\in\R_+$ such that
\[e^{-\phi}\leq C\log\Vert s_D\Vert^{-2\beta}_D,\]
then for any ample line bundle $A$ on $X$, any parabolic manifold $(M,\sigma)$ and any non-constant meromorphic map $f:M\to X$ such that $f(M)\not\subset D$, one has
\[T_{f,c_1(L,h_1)}(r)\leq T_{f,c_1(L,h_2)}(r)+O(\log T_{f,A}(r)).\]

\end{lemme}
\subsection{Lemma of logarithmic derivative}
Let us recall a standard notation from Nevanlinna theory. Given two functions \(g,h : [r_0,+\infty[ \to \R\) we write
 \[g(r)\le_{\rm exc} h(r)\]
 if there exists a Borel subset \(E\subset [r_0,+\infty[\) of finite Lebesgue measure such that \(g(r)\le h(r)\) for all \(r\in [r_0,+\infty[\setminus E\).
 Moreover, given \(s:[r_0,+\infty[\to \R\), we shall write \(g(t)\le_{\rm exc}h(t)+O(s(t))\) if there exists \(C\in \R_+\) such that \(g(t)\le_{\rm exc} h(t)+Cs(t)\).
We will use several times the following well-known calculus result:
\begin{Prop}[Borel's lemma]
Let $\phi:[r_0,+\infty[\to\R$ be a monotone increasing function.

Then for all $\epsilon>0$, one has
$$\phi'(r)\le_{\rm exc}\phi(r)^{1+\ep}.$$
\end{Prop}

\begin{proof}
\normalsize
Since $\phi$ is monotone, we know that $\phi$ is differentiable almost everywhere.
Assume $\phi\not\equiv 0$ and take $r_1\ge r_0$ such that $\phi(r_1)>0$.
Set $$E(\ep)=\{\phi'(r)>\phi(r)^{1+\ep}\}.$$
Then the Lebesgue measure of $E(\ep)$ equals
\begin{align*}
\int_{E(\ep)}1dr&\le \int_{E(\ep)}\frac{\phi'(r)}{\phi(r)^{1+\ep}}\\
&\le \int_{r_1}^{+\infty}\frac{\phi'(r)}{\phi(r)^{1+\ep}}\\
&\le \frac{1}{\ep\phi(r_1)^{\ep}}
\end{align*}

\end{proof}

Fix a $\mathscr{C}^{\infty}$ volume form $\Omega$ on $M$ (it exists since any complex manifold is orientable). There is a well-defined continuous function $v:M\to \R^+$ satisfying $\alpha^m=v\Omega$. It satisfies the following properties.
\begin{lemme}
Extend the $(m,m)$-form $\log (v) \alpha^m$ by zero on the vanishing locus of $v$. Then  $\log( v)\alpha^m$ is locally integrable on $M$ and 
the set \[S^0_{\sigma}=\{r\ge r_0;~ \log( v)\eta_r\mbox{ is not integrable}\}\]
has measure zero and does not depend on $\Omega$. 
\end{lemme}
Recall that the Ricci curvature form of $\Omega$ is defined in local coordinates $(z_1,\dots,z_n)$ such that \[\Omega=\lambda \prod_{1\leq j\leq m}\left(\frac{i}{2}dz_j\w d\bar{z_j}\right)\] by the expression
\[\Ric\Omega=-2\pi dd^c\log\lambda.\]

One defines the \textbf{Ricci function} by \[\Ric_{\sigma}(r)=2\pi\int_{M\langle r_0\rangle}\log v~\eta_{r_0}-2\pi\int_{M\langle r\rangle}\log v~\eta_{r} +T_{\Ric\Omega}(r).\]This definition is independent of the choice of the volume form $\Omega$.

\begin{Ex} If $\pi:M\to\C^m$ is a finite holomorphic covering endowed with the exhaustion $\sigma=\Vert\pi\Vert^2$, the Ricci function is given by the counting function of the ramification divisor $R_\pi$:
\[\Ric_{\sigma}(r) =-2\pi N_{R_\pi}(r).\]\end{Ex}
We shall prove the following generalisation of lemma 2.10 in \cite{BB20}:

\begin{Prop}
\label{lemme1}
Consider a pseudo-metric $h$ defined on the complement $M\setminus\Sigma$ of a thin analytic subset $\Sigma\subset M$ and the associated $(1,1)-$form $\omega$. Fix a smooth volume form $\Omega$ on $M$.
Let $\chi$ and $v$ be the two non-negative smooth functions defined over $M\setminus\Sigma$ and $M$ respectively by \[\omega^m=\chi\Omega \qquad \alpha^m=v\Omega.\] Assume that $\log\chi$ can locally be expressed as the difference of two plurisubharmonic functions.
Assume moreover that $\log\omega^m$ and $\omega\wedge \alpha^{m-1}$ are locally integrable.
Then the Ricci curvature current \[\Ric[\omega]=-2\pi dd^c[\log\lambda]\]
is well-defined and we have
$$T_{-\Ric[\omega]}(r)\le_{\rm exc}-\Ric_{\sigma}(r)+O(\log T_{[\omega]}(r)+\log r).$$
\end{Prop}
\begin{proof}
\normalsize
Denote by \[\Sigma_h=M\setminus\{z\in M\setminus\Sigma\, ;\,h_z>0 \text{ on } T_zM\}\]
the degeneracy locus of $h$ and $\widetilde{\Sigma}=\Sigma\cup\Sigma_h$.
Let $\phi=\log\chi$.
By definition, one has on $M\setminus\widetilde{\Sigma}$ the equality of currents: \[-\Ric [\omega]=2\pi dd^c[\phi]-\Ric\Omega.\] 

Applying Jensen's formula, one obtains:
\begin{align}
\label{ric}
T_{-\Ric[\omega]}(r)&=\int_{r_0}^r\frac{dt}{t^{2m-1}}\int_{M(t)}(2\pi dd^c[\phi]-\Ric\Omega)\wedge\alpha^{m-1}\\
&=2\pi\int_{M\langle r\rangle}\phi \eta_r+O(1)-T_{\Ric\Omega}(r)\\
&=2\pi m\int_{M\langle r\rangle}\log\left(\sigma^{-1}\left(\frac{\chi}{v}\right)^{1/m}\right)\eta_r+2\pi\int_{M\langle r\rangle}\log(\sigma^{m}v)\eta_r-T_{\Ric\Omega}(r)+O(1).
\end{align}
On the one hand, 
\[2\pi\int_{M\langle r\rangle}\log(\sigma^{m}v)\eta_r-T_{\Ric\Omega}(r)=-\Ric_{\sigma}(r)+O(\log r).\]

On the other hand, since the logarithm is concave, one has:
\[
\int_{M\langle r\rangle}\log\left(\sigma^{-1}\left(\frac{\chi}{v}\right)^{1/m}\right)\eta_r\le\log\left(\int_{M\langle r\rangle}\sigma^{-1}\left(\frac{\chi}{v}\right)^{1/m}\eta_r\right).
\]

A general Fubini's formula claims that for any $(2m-1)$-form $\beta$,\[\int_{M(r)}d\sigma\wedge\beta=\int_{r_0}^r\int_{M\langle t\rangle}\beta dt\]
hence
\[\frac{d}{dr}\left(\int_{M(r)}d\sigma\wedge\beta\right)=\int_{M\langle r\rangle}\beta.\]
Using this formula with $\beta=\sigma^{-1}\left(\chi/v\right)^{1/m}d^c\sigma\wedge\alpha^{m-1}$, we get
\begin{align*}
\int_{M\langle r\rangle}\sigma^{-1}\left(\frac{\chi}{v}\right)^{1/m}\eta_r&=\frac{1}{r^{2m-1}}\int_{M\langle r\rangle}\sigma^{-1}\left(\frac{\chi}{v}\right)^{1/m}d^c\sigma\wedge\alpha^{m-1}\\
&=\frac{1}{mr^{2m-1}}\frac{d}{dr}\left(\int_{M(r)}\left(\frac{\chi}{v}\right)^{1/m}\alpha^m\right)\\
&=\frac{1}{mr^{2m-1}}\frac{d}{dr}\left(\int_{M(r)}\left(\frac{\chi}{v}\right)^{-(m-1)/m}\omega^m\right)\\
&\le_{\rm exc}\frac{1}{mr^{2m-1}}\left(\int_{M(r)} \left(\frac{\chi}{v}\right)^{-(m-1)/m} \omega^m\right)^{1+\ep}\\
&\mbox{according to Borel's lemma}\\
&\le_{\rm exc}\frac{r^{(2m-1)\ep}}{m}\left(\frac{d}{dr}\left(\int_{r_0}^r\frac{dt}{t^m}\int_{M(t)} \left(\frac{\chi}{v}\right)^{-(m-1)/m} \omega^m\right)\right)^{1+\ep}\\
&\le_{\rm exc}\frac{r^{(2m-1)\ep}}{m}\left(\int_{r_0}^r\frac{dt}{t^{2m-1}}\int_{M(t)} \left(\frac{\chi}{v}\right)^{-(m-1)/m} \omega^m\right)^{(1+\ep)^2}\\
&\mbox{applying again Borel's lemma.}
\end{align*}

We would like to compare this last term to $T_{[\omega]}$.

Let $(z_1,\cdots,z_m)$ be some local coordinate system.
The complex Hessian matrix $H=\left(\frac{\dr^2\sigma}{\dr z_i\dr\bar{z_j}}\right)_{i,j}$ is invertible on $M^+$ and $\det H=v$.

Write \[\omega=\frac{i}{2}\sum_{1\leq i,j\leq m} a_{ij}dz_i\wedge d\bar{z_j}\] and denote by $A$ the matrix $(a_{ij})_{i,j}$. One has over $M^+$, \[\omega^m=\frac{\chi}{v}\alpha^m=\det(AH^{-1})\alpha^m.\]

Using the concavity of the logarithm, \[m\det(AH^{-1})^{1/m}\le Tr(AH^{-1}).\]
As a consequence, 
\begin{equation}
\label{eq:form_hadamard}
m\omega^m\le\left(\frac{\chi}{v}\right)^{(m-1)/m}\omega\wedge\alpha^{m-1}.
\end{equation}

Let us come back to the previous inequality.
One has
\begin{align*}
\int_{M\langle r\rangle}\sigma^{-1}\left(\frac{\chi}{v}\right)^{1/m}\eta_r&\le_{\rm exc} \frac{r^{(2m-1)\ep}}{m^{1+(1+\ep)^2}}T_{[\omega]}(r)^{(1+\ep)^2}
\end{align*}

Finally, \eqref{ric} implies
\begin{equation}
T_{-\Ric[\omega]}(r)\le_{\rm exc}2\pi (1+\ep)^2\log(T_{[\omega]}(r))-\Ric_{\sigma}(r)+O(\log r).
\end{equation}
\end{proof}
From proposition \ref{lemme1} together with lemma \ref{ricci}, one infers:
\begin{cor}
\label{thm:corollaire_logsing}
In the previous setup, assume that $\omega$ has at most logarithmic singularities along $\Sigma$. Then
$$T_{[-\Ric\omega]}(r)\le_{\rm exc}2\pi N_{\Sigma}(r)-\Ric_{\sigma}(r)+O(\log T_{[\omega]}(r)+\log r).$$
\end{cor}

\begin{remark}
When $M$ is a parabolic Riemann surface, the Ricci function coincides with the \emph{mean Euler characteristic} (cf \cite{PS14}, \cite{BB20})
\[\Ric_{\sigma}(r)=2\pi\mathfrak{X}_\sigma(r):=-2\pi\int_{r_0}^r\chi(M(t))\frac{dt}{t}.\]
To obtain this topological interpretation, one uses a never-vanishing global section of $K_M^\ast=TM\cong \O_M$ to construct a smooth global vector field $V$ having isolated zeroes on $M$. The Euler characteristic appears then by applying Poincaré-Hopf index theorem to the vector field $V$.

Unfortunately, this machinery cannot be used in the general setting, since the anticanonical line bundle of $M$ is not anymore trivial. In fact, one cannot even guarantee in general the existence of a non-trivial global section $\xi\in H^0(M,\det(TM))$ !   

However, if $K_M$ is indeed trivial, the Ricci function can also be related to the Euler characteristic via the Chern-Gauss-Bonnet formula.

Proposition \ref{lemme1} extends lemma 2.10 of \cite{BB20}, which we cite below.

\end{remark}
\begin{lemme}[Brotbek-Brunebarbe, \cite{BB20}, Lemma 2.10]
    Let $(M,\sigma)$ be a non-compact parabolic Riemann surface and $\xi\in H^0(M,TM)$ a global frame. Let $\Sigma$ be a reduced divisor on $M$ and $h$ be a pseudo-metric on $M\setminus \Sigma$ such that the function $\log\Vert\xi\Vert_h$ can locally be written as the difference of two subharmonic functions. Assume that the associated $(1,1)$-form $\omega$ is locally integrable. Then $\Ric[\omega]$ is well-defined and 
    \[T_{-\Ric[\omega]}(r)\lex -2\pi\mathfrak{X}_{\sigma}(r)+O(\log r+\log T_{[\omega]}(r)).\]
\end{lemme}

\section{Second main theorem for negative holomorphic sectional curvature}
\label{sec:smt}
Come back to the previous situation of a parabolic manifold $(M,\sigma)$ together with a smooth hermitian pseudo-metric $h$ on $M\setminus \Sigma$.
Assume that $h$ has non positive holomorphic bisectional curvature and holomorphic sectional curvature bounded above by a negative constant $-\gamma<0$.
\begin{lemme}
Denote by $\omega$ the $(1,1)-$form associated to the pseudo-metric $h$. Under these assumptions, one has $$[\Ric\omega]\le -\gamma[\omega]$$ on $M\setminus\Sigma$.
\end{lemme}
\begin{proof}
\normalsize
Denote by $H$ the holomorphic bisectional curvature of $h$. 
Let $x\in M\setminus \Sigma$ and take some local coordinates $(z_1,\cdots,z_m)$ centred at $x$.
Let $X\in T_xM$ be a unit holomorphic tangent vector at $x$ and take an $h-$ orthonormal basis $(e_1,\cdots,e_m)$ of $T_xM$ with $e_1=X$. 

By definition, one has:
\begin{align*}
\Ric(X,X)&=\sum_j H(X,e_j)\\
&=H(e_1,e_1)\mbox{ since $H$ is non positive}\\
&=-\gamma
\end{align*}
This completes the proof.
\end{proof}
\begin{lemme}
Keep the notations as above. Then:
\begin{enumerate}
\item $\omega$ is locally integrable on $M$.
\item Fix some local coordinate system $(z_1,\dots,z_n)$ so that
\[\omega^m=i_m\lambda\prod_{1\leq j\leq m}(dz_j\w d\bar{z_j}).\] 
Then $\log\lambda$ can be written as the difference of two plurisubharmonic functions.
\item Fix a smooth volume form $\Omega$ on $M$. It follows that writing $\omega^m=v\Omega$, $\log v$ can locally be written as the difference of two plurisubharmonic functions.
\end{enumerate}
\end{lemme}
\begin{proof}
\normalsize
\begin{enumerate}
\item Let $x$ be a point of $\Sigma$. Take a neighbourhood $W$ of $x$. Up to taking a log-resolution of the pair $(W,W\cap\Sigma)$, we may assume that $\Sigma'=W\cap\Sigma$ has simple normal crossings.
Take a local holomorphic coordinate system $(z_1,\cdots z_m)$ centred at $x$ and such that $\Sigma'=(z_1\cdots z_r=0)$. We may assume that the $z_j$ are defined in the set $\{|z_j|< 1\}$.

Denote by $\omega_0$ the following Kähler form on $W\setminus\Sigma\cong(\Delta^*)^r\times\Delta^{m-r}$:
$$\omega_0=\sum_{j=1}^r\frac{1}{|z_j|^2(\log|z_j|^2)^2}\frac{i}{2\pi}dz_j\wedge d\bar{z_j}+\sum_{j=r+1}^m\frac{i}{2\pi}dz_j\wedge d\bar{z_j}.$$
Since $\Ric\omega$ is negatively bounded, we can apply Schwarz' lemma (see \cite{Kob98}, Thm 2.4.14), and there exists a positive constant $C$ such that 
\begin{equation}\label{eq:maj}
    \omega\le C\omega_0.
\end{equation}
Therefore, $\omega$ is integrable around $x$.
\item Using the upper bound \ref{eq:maj}, we see that the function $$z\mapsto|z_1\cdots z_r|^2\lambda(z)$$ is bounded above around $x$.

On the other hand, one has $$2\pi dd^c\log(|z_1\cdots z_r|^2\lambda(z))=-\Ric\omega\ge \gamma\omega\ge 0$$ on $W\setminus\{x\}$, so that the function $\phi:z\mapsto \log(|z_1\cdots z_r|^2\lambda(z))$ is plurisubharmonic on $W\setminus\{x\}$, therefore extends as a plurisubharmonic function $\phi$ on the whole of $W$.

Hence \[\log\lambda(z)=\phi(z)-\log(|z_1\cdots z_r|^2)\] is itself a plurisubharmonic function.

\end{enumerate}
\end{proof}
Applying Corollary \ref{thm:corollaire_logsing}, we get the following analogue of Theorem 2.14 of \cite{BB20}:
\begin{Theo}
\label{thm:parab-hsc}
Let $(M,\sigma)$ be a parabolic manifold of dimension $m$ and $\Sigma\subset M$ a thin analytic subset. Consider a smooth hermitian pseudo-metric $h$ defined on $M\setminus \Sigma$ whose degeneracy set $\Sigma_h$ is thin analytic. Assume that over $M\setminus(\Sigma\cup\Sigma_h)$, the holomorphic sectional curvature of $h$ is bounded above by a negative constant $-\gamma$ and the holomorphic bisectional curvature of $h$ is non positive.  Denote by $\omega$ the $(1,1)$-form associated with $h$. Then $\omega$ is locally integrable on $M$ and we have the following Second Main Theorem-type inequality:
$$\gamma T_{[\omega]}(r)\lex 2\pi N_\Sigma(r)-\Ric_{\sigma}(r)+O(\log T_{[\omega]}(r)+\log r).$$
\end{Theo}
From this result we immediately derive a Second Main Theorem for quasi-projective varieties endowed with a metric with negative holomorphic sectional curvature.

\begin{Theo}
\label{thm:SMT-parab}
Let $X$ be a smooth projective variety of dimension $n$ and $D$ a divisor on $X$. 

Endow $X\setminus D$ with a smooth hermitian pseudo-metric $h$ whose holomorphic sectional curvature is bounded above by a negative constant $-\gamma$ and holomorphic bisectional curvature is non positive. Assume that the degeneracy set $\Sigma_h$ of $h$ is thin analytic.

Let $(M,\sigma)$ be a parabolic manifold of dimension $m\le n$ and $f:M\to X$ a non-constant meromorphic map such that $f(M)\not\subset D\cup \Sigma_h$. Denote by $\omega$ the $(1,1)-$form associated to the pseudo-metric $f^*h$ on $M\setminus f^{-1}(D)_{red}$. Then $\omega$ is locally integrable on $M$ and we have the following Second Main Theorem-type inequality:
$$\gamma T_{[\omega]}(r)\lex2\pi N^{[1]}_{f,D}(r)-\Ric_{\sigma}(r)+O(\log T_{[\omega]}(r)+\log r).$$
\end{Theo}
\section{An Arakelov-Nevanlinna inequality for parabolic manifolds (following Brotbek-Brunebarbe )}
\label{sec:proof SMT for VHS}
In \cite{BB20}, Brotbek and Brunebarbe obtained a Second Main Theorem for holomorphic maps $f:B\to X$, where $B$ is a parabolic Riemann surface and $X$ a smooth projective variety supporting a variation of complex polarised Hodge structures. This can be seen as an analogue in Nevanlinna theory of Arakelov inequalities for variations of Hodge structures.

In this section, we shall see how this generalises naturally to higher-dimensional parabolic manifolds.

This part is a simple adaptation of Brotbek-Brunebarbe's work to the higher-dimensional setting. The arguments being exactly the same, the reader is referred to the original paper \cite{BB20} for additional details and proofs.
\subsection{Variations of Hodge structures}
We briefly recall the definition of variations of complex polarised Hodge structures.

A complex polarised Hodge structure on a finite-dimensional $\C$-vector space $V$ is the data of a non-degenerate hermitian form $h$ and a decomposition $V=\bigoplus_{p\in\Z} V^p$ which is orthogonal for $h$ such that the restriction of $(-1)^ph$ is positive definite for any $p\in\Z$. The Hodge metric $h_H$ is the well-defined hermitian metric on $V$ equal to $(-1)^ph$ on each $V^p$. The length of the decomposition is the integer $w=b-a$, where $[a,b]$ is the smallest interval such that $V^p=0$ for $p\notin[a,b]$.

The Hodge filtration $F^\bullet$ on $V$ is the natural decreasing filtration defined by $F^p=\bigoplus_{q\geq p}V^q$. The Hodge decomposition is entirely determined by $F^\bullet$ since $V^p=F^p\cap (F^{p+1})^\bot$.

\begin{Def}
   A variation of complex polarised Hodge structures $\mathbb{V}$ on a complex analytic space $S$ is the data of a complex local system $\mc{L}$ equipped with a non-degenerate hermitian form $h:\mc{L}\otimes_\C \overline{\mc{L}}\to\underline{\C}_S$ and a locally split finite filtration $\mc{F}^\bullet$ of $\mc{V}:=\mc{L}\otimes_\C\O_S$ by analytic coherent subsheaves such that
\begin{itemize}
 \item for every $s\in S$, the triple $(\mc{L}_s,\mc{F}_s^\bullet,h_s)$ defines a complex polarised Hodge structure;
\item The connection $\nabla:=\mathrm{id}\otimes d$ satisfies \emph{Griffith's transversality property}: $\nabla(\mc{F}^p)\subset\mc{F}^{p-1}\otimes_{\O_S}\Omega^1_S$ for all $p$.
\end{itemize}

When $S$ is smooth (which will be the case for us), the data of the complex local system $\mc{L}$ is equivalent to the data of the holomorphic vector bundle $\mc{V}:=\mc{L}\otimes_\C\O_S$ equipped with the integrable connection $\nabla:=\mathrm{id}\otimes d$.
The length of the complex polarized Hodge structure on the complex vector space $\mc{L}_s$ is independent of the point $s\in S$ and is called the length of $\mathbb{V}$.
\end{Def}

Let $(\mathcal{V}, \nabla, \mc{F}^{\bullet}, h)$ be a variation of complex polarized Hodge structures on a complex manifold $S$, and define $\mc{E}:= \bigoplus_p \mc{E}^p$ with $\mc{E}^p := \mc{F}^p / \mc{F}^{p+1}$. It follows from Griffiths transversality that the connection $\nabla$ induces an $\O_S$-linear morphism  $ \mc{E}^p \rightarrow  \Omega^1_S \otimes_{\O_S} \mc{E}^{p-1}$ for every $p$.
We denote by $\phi_p$ the corresponding element of $ \Omega^1_S \otimes_{\O_S} \mathcal{H}om(\mc{E}^p , \mc{E}^{p-1})$ and define the Higgs field $\phi := \oplus_p \phi_p \in \Omega^1_S (\mathcal{E}nd(\mc{E}))$, which can be viewed as an $\O_S$-linear morphism $ TS \rightarrow \mathcal{E}nd(\mc{E})$. The pair $(\mc{E}, \phi)$ together with the decompositions $\mc{E} := \bigoplus_p \mc{E}^p$ and $\phi := \oplus_p \phi_p $ is called the associated system of Hodge bundles. \\

The holomorphic vector bundle $\mc{E}$ comes equipped with the positive-definite hermitian metric $h_H$, the Hodge metric, which makes the decomposition $\mc{E} := \bigoplus_p \mc{E}^p$ orthogonal.

\subsection{The Griffiths-Schmid pseudo-metric}

Let $S$ be a complex manifold endowed with a variation of complex polarized Hodge structures $(\mathcal{V}, \nabla, \mc{F}^{\bullet}, h)$ of length $w$. We will construct a metric on $S$ having negative holomorphic sectional curvature and non-positive holomorphic bisectional curvature.

\begin{Def}
For any integer $p$, the pull-back by $\phi_p$ of the Hodge metric on $\mathcal{H}om(\mc{E}^p , \mc{E}^{p-1})$ induces a pseudo-metric $h_p$ on $TS$. The sum of these pseudo-metrics, which coincides with the pseudo-metric on $ T_S$ obtained as the pull-back by $\phi$ of the Hodge metric on $\mathcal{E}nd(\mc{E})$, is called the Griffiths-Schmid pseudo-metric on $S$ induced by $\V$.
\end{Def}

By definition, the locus where the Griffiths-Schmid pseudo-metric is non-degenerate coincide with the Zariski-open subset of $S$ where the $\O_S$-linear morphism $\phi : T_S \rightarrow \mc{E} nd(\mc{E})$ is injective, or in other words with the locus where the period map associated to $\V$ is immersive.\\

\begin{Prop}[Griffiths, cf. \cite{Schmid}]
The curvature form of the Chern connection of the Griffiths line bundle $L_{\mathbb{V}} :=  \otimes_p \det \mc{F}^p$ equipped with the Hodge metric corresponds to the Griffiths-Schmid pseudo-metric.

In other words, the K\"ahler form of the Griffiths-Schmid pseudo-metric on $S$ equals the Chern curvature form of $(L_{\mathbb{V}}  , h_H)$.
\end{Prop}
\begin{Prop}[\cite{BB20} Proposition 3.5]
\label{curv-griffiths}
Over the Zariski-open subset where it is non-degenerate, the Griffiths-Schmid metric has non-positive holomorphic bisectional curvature and negative holomorphic sectional curvature $-  \gamma$ with
\[ \frac{1}{\gamma}  \leqslant \frac{w^2}{4} \cdot \rank( \mc{V}). \]
\end{Prop}

\subsection{A version of the Second Main Theorem}

In view of Proposition \ref{curv-griffiths}, we may apply Theorem \ref{thm:parab-hsc} to the Griffiths-Schmid pseudo-metric and get the following.
\begin{Theo}\label{thm:Arakelov-Nevanlinna inequality}
Let $(M,\sigma)$ be a parabolic manifold, $\Sigma \subset M$ a thin analytic subset and $\mathbb{V} = (\mc{L}, \mc{F}^{\bullet}, h)$ be a variation of complex polarised Hodge structures of length $w$ on $ M\setminus\Sigma$ with a non-constant period map. Then the first Chern form of the holomorphic line bundle $ L_{\mathbb{V}} = \otimes_p \det  \mc{F}^p$ equipped with the hermitian metric induced by $h$ extends as a current \( [L_{\mathbb{V}}  ] \) on $M$, and one has
 \[ T_{[L_{\mathbb{V}} ]}(r) \lex \frac{w^2 \cdot \rk \mc{L}}{4}  \left( - \frac{1}{2\pi}\Ric_\sigma(r) +  N_{\Sigma}(r)\right)+  O( \log T_{[L_{\mathbb{V}}]}(r)+ \log r). \]

\end{Theo}

Putting together theorem \ref{thm:Arakelov-Nevanlinna inequality} and proposition \ref{thm:griff-ext} below, we get
\begin{Theo}\label{thm:SMT for VHS}
Let $X$ be a smooth projective complex algebraic variety of dimension $n$ and $\mathbb{V} = (\mc{L}, \mc{F}^{\bullet}, h)$ be a variation of complex polarised Hodge structures defined on the complement of a normal crossing divisor $D \subset X$. Assume that $\mc{L}$ has unipotent monodromies around the irreducible components of $D$. We denote by $\bar{\mc{F}^p}$ the canonical Deligne-Schmid extension of $\mc{F}^p$ to $X$ for any integer $p$ and by $\bar{L}_{\mathbb{V}} = \otimes_p \det \bar{\mc{F}^p}$ the canonical extension of the Griffiths line bundle of $\mathbb{V}$.
Let $(M,\sigma)$ be a parabolic manifold of dimension $m\le n$ and $f:M\to X$ be a non-constant holomorphic map such that \(f(M)\not\subset D\). Then the first Chern form of the line bundle $ L_{\mathbb{V}}$ equipped with the hermitian metric induced by $h$ extends as a current \( [L_{\mathbb{V}}  ] \) on $M$, and for any ample line bundle $A$ on $X$, one has
  \[T_{f, \bar{L}_{\mathbb{V}}}(r)  \lex \frac{w^2 \cdot \rk \mc{L}}{4}  \left( - \frac{1}{2\pi}\Ric_\sigma(r) +  N_{\Sigma}(r)\right) +O(\log r+\log( T_{f,A}(r))).\]

\end{Theo}
The following fact is a simple consequence of lemma \ref{thm:lemme2}, since the assumption on the monodromy ensures that the Hodge metric satisfies the growth condition.
\begin{Prop}\label{thm:griff-ext}
  With the above notation,
  one has
  \[ T_{f, \bar{L_{\mathbb{V}} }}(r)  \leqslant T_{[L_{\mathbb{V}} ]}(r) +O(\log( T_{f,A}(r))).\]
\end{Prop}
\section{The parabolic tautological inequality}\label{sec:tauto}

Consider a holomorphic map $f:M\to X$, where $(M^m,\sigma)$ is a parabolic manifold and $X$ a smooth projective variety of dimension $n\geq m$. For any normal crossing divisor $D\subset X$ such that $f(M)\not\subset D$, one wants to estimate the characteristic function of the \emph{tautological line bundle} $\O(1)$ on the projectivisation of the $m$-th power of the (logarithmic) cotangent bundle $\P(\Omega^m_X)$ (or $\P(\Omega_X^m(\log D))$), using other Nevanlinna functions. This is an example of a so-called \emph{tautological inequality}.

This was achieved in both the compact and logarithmic cases for $M=\C$ first (McQuillan \cite{McQ98}, Vojta \cite{Voj00}), then $M=\C^m$ (Gasbarri-Pacienza-Rousseau \cite{GPR13}).

In this section, we generalise some of the results of \cite{GPR13} to general parabolic manifolds. More specifically, we are interested in extending the following.
\begin{Prop}[Theorem 3.3 in \cite{GPR13}]
    Let $X$ be a smooth projective variety, $D\subset X$ a simple normal crossings divisor, $A$ an ample line bundle on $X$ and $f:\C^m\to X$ a holomorphic mapping of maximal rank such that $f(\C^m)\not\subset D$. Set $X_1=\P\left(\bigwedge^m\clog{X}{D}\right)$ and $L=\O_{X_1}(1)$. Denote by $f_1$ the natural meromorphic lift $\C^m\to X_1$ of $f$.
    Then 
    \[T_{f_1,L}(r)\lex N^{[1]}_{f,D}(r)+O(\log T_{f,A}(r)).\]
\end{Prop}

\subsection{The parabolic setting}
Our situation is as follows. Let $X$ be a smooth $n$-dimensional projective variety, $D=D_1+...+D_c$ a simple normal crossing divisor. We consider the projectivisation 
\[X_1=\P\left(\bigwedge^m\clog{X}{D}\right)\] of the vector bundle of logarithmic $m$-forms on $X$. 

Let $(M,\sigma)$ be a parabolic manifold of dimension $m$ together with a holomorphic map $f:M\to X$ such that $f(M)\not\subset D$. Denote $H=(f^{-1}D)_{red}$.

Assume moreover that $m\le n$ and $f$ has maximal rank $m$, where we define the rank of a holomorphic map to be 
\[\rk f=\max\{\rk d_xf; x\in M\}.\]
The subset \[\Sigma=\{x\in M;\rk d_xf<m\}\] is analytic of codimension at least 2 (see e.g. \cite{AS71}, lemma 1.14). 

The morphism of sheaves 
\[f^\ast\bigwedge^m\clog{X}{D}\longrightarrow\bigwedge^m\clog{M}{H}\]
induces a rational map \[f':M\dashrightarrow \P\left(\bigwedge^m\clog{X}{D}\right)\]lifting $f$.
The meromorphic map $f'$ is well-defined and holomorphic outside an analytic subset $S$ of codimension at least $2$, which is the union of $\Sigma$ and the locus where $H$ is not normal crossings.

The restriction of $f'$ to $M\setminus (S\cup H)$ is simply given by differentiating $f$. In local coordinates $t_1,\dots,t_m$, we have 
\[f'|_{M\setminus (S\cup H)}(t_1,\dots,t_m)=\left(f(t_1,\dots,t_m),\left[\frac{\dr f}{\dr t_1}\w\cdots\w\frac{\dr f}{\dr t_m}\right]\right)\]

We shall prove the following estimation.
\begin{Theo}
\label{thm:tauto}
For any ample line bundle $A$ on $X$, one has
\[T_{f',\O_{X_1}(1)}(r)\lex N_{f,D}^{[1]}(r)-\frac{1}{2\pi}\Ric_\sigma(r)+O\left(\log T_{f,A}(r)\right).\]
\end{Theo}
\begin{proof}
\normalsize
Let $A$ be an ample line bundle on $X$ and $\omega_A$ the positive curvature form of an appropriate hermitian metric on $A$.

For any $i=1,\dots,c$, let $s_i$ be a holomorphic section of $\O_X(D_i)$ defining $D_i$, and endow each $\O_X(D_i)$ with a smooth hermitian metric $\kappa_i$ such that $|s_i|_{\kappa_i}<1$. Denote $s=\prod_{i=1}^cs_i$ and $|s|=\prod_i |s_i|_{\kappa_i}$.

We can define a smooth metric on $T_X(-\log D)$ by the singular $(1,1)$-form
\[\omega_1=\omega_A+\sum_{1\leq i\leq c}\frac{d|s_i|_{\kappa_i}\w d^c| s_i|_{\kappa_i}}{| s_i|_{\kappa_i}^2}.\]

The form $\omega_1^m$ induces a smooth metric $\tilde{h}$ on $\O_{X_1}(-1)$. 
Let $\tilde{\omega}$ be the smooth pseudo-metric $f^\ast\omega_1$ on $M\setminus (S\cup H)$, with logarithmic singularities along $H$. Moreover, the form $\log(\tilde{\omega}^m)$ is locally integrable along $H$: writing $\omega_1^m=\lambda(z)\prod_{j=1}^nidz_j\w d\bar{z}_j$ in some local coordinates $z_1,\dots,z_n$, the function $\log(|s|^2\lambda)$ is locally bounded and $\log(|s|^2)$ is plurisubharmonic. Hence $\log(\lambda)=\log(|s|^2)-\log(|s|^2\lambda)$ is locally integrable along $D$.

By definition, we have \[-\Ric \tilde{\omega}=2\pi(f')^\ast c_1(\O_{X_1}(1),\tilde{h}^{-1}),\]
on $M\setminus (S\cup H)$, so that 
\[2\pi T_{f',\O_{X_1}(1)}=T_{[-\Ric \tilde{\omega}]}.\]
Define $\tilde{\omega}_s=\frac{1}{(\log|s\circ f|^2)^2}\tilde{\omega}$. 
It is clear that $\tilde{\omega}_s$ has at most logarithmic singularities along $H$. Thus we may apply \ref{thm:corollaire_logsing}:
\[T_{[-\Ric\tilde{\omega}_s]}\lex2\pi N_H(r)-\Ric_\sigma(r)+O(\log T_{[\tilde{\omega}_s]}+\log r).\]
Since $N_H(r)=N_{f,D}^{[1]}(r)$, it suffices to prove the following inequalities.
\begin{align}
 \label{eq:taut_1}   T_{[-\Ric\tilde{\omega}]}&\leq T_{[-\Ric\tilde{\omega}_s]}+O(\log r+\log T_{f,A})\\
  \label{eq:taut_2}   \log T_{[\tilde{\omega}_s]}&\leq O(\log T_{f,A})
\end{align}
Let us first establish \eqref{eq:taut_1}. Over $M\setminus(H\cup S)$, one has
\begin{equation}
\label{eq:comparaison}
-\Ric\tilde{\omega}_s=-\Ric\tilde{\omega}-2\pi dd^c\log(\log(|s|^2)^{2m})=-\Ric\tilde{\omega}-2\pi mdd^c\log(\log(|s|^2)^{2}).
\end{equation}
Recall that we have chosen the metrics on $\O(D_i)$ such that $|s|<1$. Hence 
\[\log(\log(|s|^2)^2)=2\log(-\log(|s|^2)).\]
Because of the moderate singularities of $\log(-\log(|s|^2))$, it is readily verified that we can apply Jensen's formula to this function. This gives
\begin{align*}
T_{[dd^c\log(\log(|s|^2)^2)]}(r)&=\int_{M\langle r\rangle}\log(\log(|s|^2)^2)d\mu_r+O(1)\\
&\leq 2\log\left(\int_{M\langle r\rangle}-\log(|s|^2)^2)d\mu_r\right)+O(1)\\
&\leq 2\log(-T_{[dd^c\log|s|^2]}(r))+O(1)\\
&\leq 2\log\left(-\sum_{i=1}^cN_{f,D_i}(r)+T_{f,\O_X(D)}(r)\right)+O(1)\\
&\leq 2\log T_{f,\O_X(D)}(r)+O(1)\leq O(\log r+\log T_{f,A})
\end{align*}
Therefore, integrating \eqref{eq:comparaison} implies \eqref{eq:taut_1}.

By definition,
\begin{align*}
\tilde{\omega}_s&=\frac{1}{(\log|s\circ f|^2)^2}f^\ast\omega_A+f^\ast\sum_{i=1}^c\frac{d|s_i|_{\kappa_i}\w d^c|s_i|_{\kappa_i}}{|s_i|_{\kappa_i}^2(\log|s_i|_{\kappa_i}^2)^2}\\
&\leq f^\ast\omega_A+f^\ast\sum_{i=1}^c\frac{d|s_i|_{\kappa_i}\w d^c|s_i|_{\kappa_i}}{|s_i|_{\kappa_i}^2(\log|s_i|_{\kappa_i}^2)^2}
\end{align*}

By a straightforward computation, one has
\[\frac{d|s_i|_{\kappa_i}\w d^c|s_i|_{\kappa_i}}{|s_i|_{\kappa_i}^2(\log|s_i|_{\kappa_i}^2)^2}=-\frac{1}{4\log(|s_i|_{\kappa_i}^2)}dd^c\log(|s_i|_{\kappa_i}^2)-\frac{1}{8}dd^c\log(\log(|s_i|_{\kappa_i}^2)^2).\]
Integrating this relation yields 
\[T_{\left[f^\ast\frac{d|s_i|_{\kappa_i}\w d^c|s_i|_{\kappa_i}}{|s_i|_{\kappa_i}^2(\log|s_i|_{\kappa_i}^2)^2}\right]}\leq O(T_{f,\O_X(D_i)}+T_{f,A})\leq O(T_{f,A}),\]
and therefore
\[ \log T_{[\tilde{\omega}_s]}\leq O(\log T_{f,A}).\]
\end{proof}

\bibliography{bib}{}
\bibliographystyle{plain}
\end{document}